\documentclass[reqno, 12pt]{amsart}

\usepackage{amsmath,amssymb,amsthm,amsfonts}
\usepackage{fullpage}
\usepackage{xcolor}
\usepackage[colorlinks=true,linkcolor=blue,anchorcolor=blue,citecolor=red]{hyperref}
\allowdisplaybreaks 
\usepackage{cite}
\usepackage{colonequals}

\newtheorem{theorem}{Theorem}[section]
\newtheorem{lemma}[theorem]{Lemma}

\newtheorem{proposition}[theorem]{Proposition}

\theoremstyle{definition}

\theoremstyle{remark}
\newtheorem{remark}[theorem]{Remark}

\newcommand{\Z}{\mathbb{Z}}

\newcommand{\C}{\mathbb{C}}
\renewcommand{\P}{\mathbb{P}}

\newcommand{\F}{\mathbb{F}}

\newcommand{\cA}{\mathcal{A}}
\newcommand{\cB}{\mathcal{B}}

\newcommand{\on}{\operatorname}
\renewcommand{\Re}{\on{Re}}

\renewcommand{\mod}[1]{\,(\mathrm{mod}\,#1)}
\newcommand{\of}[1]{\left(#1\right)}
\newcommand{\set}[1]{\left\{#1\right\}}

\author{Biao Wang}
\address{School of Mathematics and Statistics, Yunnan University, Kunming, Yunnan 650091, China}
\email{bwang@ynu.edu.cn}
\date{\today}

\makeatletter
\@namedef{subjclassname@2020}{\textup{}2020 Mathematics Subject Classification}
\makeatother

\title[Erd\H{o}s covering systems]{Bounds for Erd\H{o}s covering systems in global function fields}

\begin{document}
	
\begin{abstract}
Covering systems of the integers were introduced by Erd\H{o}s in 1950. Since then, many beautiful questions and conjectures about these objects have been posed.  Most famously, Erd\H{o}s asked whether the minimum modulus of a covering system with distinct moduli is arbitrarily large. This problem was resolved in 2015 by Hough, who proved that the minimum modulus is bounded. In 2022, Balister et al. developed Hough's method and gave a simpler but more versatile proof of Hough's result. Their technique has many applications in a number of variants on  Erd\H{o}s' minimum modulus problem. In this paper, we show that there is no covering system of multiplicity $s$ in any global function field of genus $g$ over $\mathbb{F}_q$ for $q\geq(82.26+18.88g)e^{0.95g}s^2$. Moreover, we obtain that there is no covering system of $\mathbb{F}_q[x]$ with distinct moduli for $q>73$. This improves the results in the previous work of the author joint with Li, Wang and Yi.
\end{abstract}

\maketitle

\section{Introduction}\label{sec_intro}

Introduced by Erd\H{o}s \cite{Erdos1950} in 1950, a \textit{covering system} of the integers $\Z$ is a finite collection of infinite arithmetic progressions whose union is $\Z$. In the same article, Erd\H{o}s asked whether the minimum modulus of a covering system with distinct moduli is arbitrarily large. This is the so-called \textit{minimum modulus problem}. On this problem, Erd\H{o}s wrote ``This is perhaps my favorite problem" in \cite[page 166]{Erdos1995} and offered \$1000 for a proof or disproof \cite[page 467]{BBH1990}. In 2015, building on Filaseta, Ford, Konyagin, Pomerance and Yu's work \cite{FFSPY2007}, Hough \cite{Hough2015}  resolved this problem, showing that the minimum modulus is at most $10^{16}$ in any covering system of $\Z$ with distinct moduli. In 2022, Balister, Bollob\'as, Morris, Sahasrabudhe and Tiba \cite{BBMST2022} reduced Hough's bound to $616,000$. They gave a simpler and stronger variant of Hough's method, which is called the \textit{distortion method} \cite{BBMST2020_2}. This method has many applications in a number of variants on  Erd\H{o}s' minimum modulus problem. For example, using the distortion method, Cummings, Filaseta and Trifonov \cite{CFT2022} reduced Balister et al.'s bound to 118 for covering systems with distinct squarefree moduli. And Klein, Koukoulopoulos and Lemieux \cite{KKL2022} generalized Hough's result to covering systems of multiplicity. We refer readers to Balister's recent summary \cite{Balister2024} of these beautiful results. 

After these work, some progress on Erd\H{o}s covering systems in number fields, finite fields and global function fields has been made. Let $\F_q$ be a finite field of $q$ elements. Recently, Li, Wang and Yi \cite{LiWangYi2023pre} provided a solution to a generalization of Erd\H{o}s' minimum modulus problem in number fields. Then the authors in \cite{LiWangWangYi2023pre} resolved an analogue of Erd\H{o}s' minimum modulus problem for polynomial rings over finite fields in the degree aspect. This disproves a conjecture by Azlin in his master's thesis \cite{Azlin2011thesis}. Later, they \cite{LiWangWangYi2024} proved that there does not exist any covering system of $\mathbb{F}_q[x]$ with distinct moduli for $q\geq 759$. Indeed, they gave an upper bound on $q$ for covering systems in global function fields with a constant field $\F_q$. In this paper, we will improve the results in \cite{LiWangWangYi2024} by taking suitable parameters in the distortion method. Our main results are stated as follows.

Let $K$ be a global function field in one variable with a finite constant field $\F_q$. Let $g$ be the genus of $K$. 	Let $S$ be any non-empty finite set of primes of $K$, and let $O_S$ be the ring of $S$-integers in $K$. Then $O_S$ is a Dedekind domain. A \textit{covering system} of $O_S$ is a finite collection  $\mathcal{A}=\{A_{i}\colon 1\leq i\leq m\}$ of congruences in $O_S$ whose union is $O_S$, where $A_i=a_i+I_{i}$ for some $a_i\in O_S$ and some ideal $I_{i}\subseteq O_S$. 
Let 
\[
    m(\cA)=\max_{ \text{ideals } I \subset O_S}\#\{1\leq i\leq m\colon I_i=I\}
\]
be the multiplicity of $\cA$. 

\begin{theorem}\label{mainthm_gff}
Given an integer $s\geq1$. Let $K$ be a global function field with a constant field $\F_q$. Let $g$ be the genus of $K$.  Let $S$ be any non-empty finite set of primes of $K$, then there does not exist any covering system of $O_S$ with multiplicity $s$ for $q\geq (82.26+18.88g)e^{0.95g}s^2$. 
\end{theorem}

\begin{remark}
	The bound in the previous result \cite[Theorem 1.1]{LiWangWangYi2024} is $(1.14+0.16g)e^{6.5+0.97g}s^2$, which is approximately equal to $(758.26+106.42g)e^{0.97g}s^2$. Hence this bound is reduced by about $89\%$ for $g=0$ in Theorem~\ref{mainthm_gff}.
\end{remark}

If $O_S=\F_q[x]$ is a polynomial ring over $\F_q$, then the genus $g=0$ in this case. If the moduli of a covering system is distinct, then the multiplicity $s=1$. As a corollary of Theorem~\ref{mainthm_gff}, we obtain that there does not exist any covering system of $\F_q[x]$ with distinct moduli for $q\geq 83$. By slightly modifying the proof of Theorem~\ref{mainthm_gff}, we get a better result as follows.

\begin{theorem}\label{mainthm_ff}
There is no covering system of $\F_q[x]$ with distinct moduli for $q>73$.
\end{theorem}

For the proof of Theorem~\ref{mainthm_gff}, we follow the approaches in the works \cite{KKL2022, LiWangYi2023pre} and apply the distortion method as in \cite{LiWangWangYi2024}. In our previous work \cite{LiWangWangYi2024}, we take $\delta_1=\cdots=\delta_J=0.5$ in the distortion method, and the referee pointed out that this is not sharp. In this paper, by modifying the upper bound on the second moments, we take a  suitable choice of the $\delta_j$'s to get a better result. In particular, the second moments for prime ideals of degrees $1$ and $2$ will be the main part. These analysis helps us to get better bounds in Theorem~\ref{mainthm_gff} than \cite[Theorem 1.1]{LiWangWangYi2024} .

In Section~\ref{sec_pre}, we review the Riemann hypothesis for function fields and introduce Klein et al.'s modification of the distortion method in the setting of Dedekind domains. Then in Section~\ref{sec_second_moments}, we provide an explicit upper bound for the summation of the second moments. At the end, we use the upper bound on the number of prime ideals in $O_S$ to give a proof of Theorems~\ref{mainthm_gff} and \ref{mainthm_ff} in the last section.

\section{Preliminaries}\label{sec_pre}

\subsection{The Riemann hypothesis for function fields}

In this subsection, we review some facts in global function fields and refer readers to Rosen's book \cite[Chapter 5]{Rosen2002} for more details. Let $K$ be a global function field in one variable with a constant field $\F_q$. Suppose the genus of $K$ is $g$. Let $\zeta_K(s)=\sum_{A\ge0}|A|^{-s}$ be the zeta function of $K$, where the summation runs over all effective divisors $A$ of $K$, and $|A|$ denotes the norm of $A$. Then there is a polynomial $L_K(u)\in\Z[u]$ of degree $2g$ such that
 $$\zeta_K(s)=\frac{L_K(q^{-s})}{(1-q^{-s})(1-q^{1-s})}$$
 for $\Re s>1$. This gives an analytic continuation of $\zeta_K(s)$ to $\C$, and $\zeta_K(s)$ has simple poles at $s=0,1$. Setting $\xi_K(s)=q^{(g-1)s}\zeta_K(s)$, one has the functional equation $\xi_K(1-s)=\xi_K(s)$ for all $s$. Let $\rho_1,\dots, \rho_{2g}$ be the inverse roots of $L_K(u)$, i.e., 
 $$L_K(u)=\prod_{i=1}^{2g}(1-\rho_i u).$$
 
 By the celebrated work of Weil \cite{Weil1948}, the following Riemann hypothesis holds for the zeta function of function fields.
 
 \begin{theorem}[The  Riemann hypothesis for function fields]\label{thm_RH_ff}
 	All the roots of $\zeta_K(s)$ lie on the line $\Re s=1/2$. Equivalently, the inverse roots of $L_K(u)$ all have absolute value $\sqrt{q}$.
 \end{theorem}
 	
 Let $\pi_K(n)$ be the number of primes of degree $n$ in $K$. Then the following analogue of the prime number theorem \cite[Theorem 5.12]{Rosen2002} holds for function fields:
\begin{equation}\label{eqn_pnt}
	\pi_K(n)=\frac{q^n}{n}+O\of{\frac{q^{n/2}}{n}}.
\end{equation}
Indeed, by the proof of \cite[Theorem 5.12]{Rosen2002}, we have the following relation between $\pi_K(n)$ and the inverse roots $\rho_i$
\begin{equation}\label{eqn_pnt_gff}
	\sum_{d|n} d\pi_K(d)=q^n+1-\sum_{i=1}^{2g}\rho_i^n
\end{equation}
 for all $n\ge1$. Then \eqref{eqn_pnt} is a consequence of Theorem~\ref{thm_RH_ff} by \eqref{eqn_pnt_gff}.

Let $S$ be any non-empty finite set of primes of $K$ and  $O_S$ be the ring of $S$-integers in $K$.  Let $\pi_{K,S}(n)$ be the number of prime ideals of degree $n$ in $O_S$. Then $\pi_{K,S}(n)\leq \pi_K(n)$. Thus,  by \eqref{eqn_pnt_gff} we get an explicit upper bound of $\pi_{K,S}(n)$
\begin{equation}\label{eqn_pi_upper_bound}
 	 \pi_{K,S}(n)\leq \frac{q^n+1}{n}+\frac{2gq^{n/2}}{n}
\end{equation}
 for all $n\ge1$. This estimation will be used in the proof of Theorem~\ref{mainthm_gff} to bound the second moments in the distortion method.

\subsection{The distortion method}\label{sec_distortion_method}
In this subsection, following Klein et al.'s modification \cite{KKL2022}, we will briefly introduce the distortion method in the setting of Dedekind domains as illustrated in \cite{LiWangWangYi2024}. 

Let $R$ be a Dedekind domain satisfying that the quotient ring $R/I$ is finite for any nonzero ideal $I$ in $R$. Let  $|I|\colonequals |R/I|$ be the norm of $I$. Let $\mathcal{A}=\{A_{i}\colon 1\leq i\leq m\}$ be a finite collection of congruences in $R$, where $A_i=a_i+I_i, a_i\in R$ and $I_i$ is a nonzero ideal in $R$. Let
$$m(\cA)=\max_{ \text{Ideals } I \subset R}\#\{1\leq i\leq m\colon I_i=I\}$$
be the multiplicity of $\cA$.

Let $Q=I_1\cap\cdots\cap I_m$. Let  $Q=\prod_{i=1}^J P_i^{\nu_i}$ be the unique prime ideal decomposition of $Q$, where the $P_i$'s are distinct prime ideals and $|P_1|\leq |P_2|\leq \cdots \leq |P_J|$. Set $Q_j=\prod_{i=1}^jP_i^{\nu_i}, 1\leq j \leq J$ and $Q_0=R$.  For $1\le j \le J$, define 
\[
\cB_j\colonequals\bigcup_{\substack{1\le i\le m\\ I_i|Q_j, I_i\nmid Q_{j-1}}} \set{a+Q\colon a \equiv a_i \mod{I_i}}.
\]

Let $\pi_j\colon R/Q \to  R/Q_j$ be the natural projection, $0\leq j\leq J$. For any $a\in R$, we put $\bar{a}=a+Q\in R/Q$ and let 
\[
F_j(\bar{a})=\set{\bar{r}\in R/Q\colon \pi_j(\bar{r})=\pi_j(\bar{a})}
\]
be the fibre at $\bar{a}$. For any $\bar{a}\in R/Q$, let
\[
\alpha_j(\bar{a})=\frac{\# (F_{j-1}(\bar{a})\cap\cB_j)}{\#F_{j-1}(\bar{a})}
\]
be the portion of $\cB_j$ in the fibre $F_{j-1}(\bar{a})$.

Now, we inductively define a sequence of probability measures $\P_0, \P_1, \dots, \P_J$ on $R/Q$. Let $\P_0$ be the uniform probability measure. Let $\delta_1,\dots,\delta_J\in[0,\frac12]$. Assuming $\P_{j-1}$ is defined, we define $\P_j$ as follows:
\[
\P_j(\bar{a})\colonequals \P_{j-1}(\bar{a}) \cdot \left\{
\begin{aligned}
	&\frac{1_{\bar{a}\notin \cB_j}}{1-\alpha_j(\bar{a})} & \text{if }  \alpha_j(\bar{a})<\delta_j,\\
	&\frac{\alpha_j(\bar{a})-1_{\bar{a}\in \cB_j}\delta_j}{\alpha_j(\bar{a})(1-\delta_j)} & \text{if }  \alpha_j(\bar{a})\geq \delta_j.
\end{aligned}
\right.
\]

For any integer $k\geq1$, let
\begin{equation}\label{eqn_kth_moment_definition}
    M_j^{(k)}\colonequals\sum_{\bar{a}\in R/Q}\alpha_j^k(\bar{a})\P_{j-1}(\bar{a}).
\end{equation}
be  the $k$-th moment of $\alpha_j$. By Balister et al.'s work \cite[Theorem~3.1]{BBMST2022} and Klein et al.'s work \cite{KKL2022}, the key result of the distortion method used in this paper is the following.

\begin{theorem}\label{thm_distortion_method}
With the notation as above,	for any $\delta_1,\dots,\delta_J\in[0,\frac12]$, if
	\[
	\sum_{j=1}^J\frac{M_j^{(2)}}{4\delta_j(1-\delta_j)}<1,
	\]
	then $\cA$ does not cover $R$. 
\end{theorem}

 \section{Bounding the second moments}\label{sec_second_moments}
 
 In this section, we proceed to bound the second moments. In \cite[Lemma 3.4]{KKL2022} and \cite[Lemma 2.4]{LiWangWangYi2024}, the parameters $\delta_j$ are taken to be $\frac12$ for all $j$ during the estimation on $M_j^{(2)}$. Instead of doing so, we keep the parameters $\delta_j$ in the upper bound in the following estimate, which is similar to \cite[Theorem 3.2]{BBMST2022} and \cite[Corollary 4.6]{CFT2022}. This will help us find a suitable choice of $\delta_j$ in the proof of Theorem~\ref{mainthm_gff}.
 
\begin{lemma}\label{lem_second_moments}
For the second moments $M_j^{(2)}$, $1\leq j\leq J$ in Theorem~\ref{thm_distortion_method}, we have
\begin{equation}\label{lem_second_moments_eqn}
	 M_j^{(2)}\leq \frac{s^2}{(|P_j|-1)^2}  \exp\set{\sum_{i<j}\frac1{1-\delta_i}\Big(\frac{3}{|P_i|}+h(|P_i|)\Big)},
\end{equation}
 where $h(t)=\frac{5t-3}{t(t-1)^2}$ for $t\ge2$.
\end{lemma}

\begin{proof} 
By the proof of \cite[Lemma 2.4]{LiWangWangYi2024}, we have
\begin{equation}\label{lem_second_moments_eqn1}
	M_j^{(2)}\leq s^2 \sum_{1\leq r_1, r_2\leq \nu_j} \sum_{H_1, H_2 \mid Q_{j-1} } \frac{1}{|H_1\cap H_2|\cdot|P_j|^{r_1+r_2}}\cdot\prod_{\substack{i\leq j-1\\ P_i\mid H_1 \cap H_2 }}(1-\delta_i)^{-1}.
\end{equation}

The summation over $r_1$ and $r_2$ is bounded by
\begin{equation}\label{lem_second_moments_eqn2}
	\sum_{1\leq r_1, r_2\leq \nu_j} \frac1{|P_j|^{r_1+r_2}} \leq \frac{1}{(|P_j|-1)^2}.
\end{equation}

As regards the summation over $H_1$ and $H_2$, notice that the product $\prod_{\substack{i<j\\ P_i\mid H }}(1-\delta_i)^{-1}$ is a multiplicative function on ideals $H$ in $R$. It follows that
\begin{align}
	&\sum_{H_1, H_2 \mid Q_{j-1} } \frac{1}{|H_1\cap H_2|}\cdot\prod_{\substack{i\leq j-1\\ P_i\mid H_1 \cap H_2 }}(1-\delta_i)^{-1} \nonumber\\
	&\leq \prod_{i<j}\left(1+\sum_{\nu\ge1}\frac{2\nu+1}{(1-\delta_i)|P_i|^\nu}\right) \nonumber\\
	&= \prod_{i<j}\left(1+\frac1{1-\delta_i}\Big(\frac{3}{|P_i|}+\frac{5|P_i|-3}{|P_i|(|P_i|-1)^2}\Big)\right) \nonumber\\
 	&\leq  \exp\set{ \sum_{i<j}\frac1{1-\delta_i}\Big(\frac{3}{|P_i|}+h(|P_i|)\Big)}. \label{lem_second_moments_eqn3}
\end{align}

Thus, \eqref{lem_second_moments_eqn} follows by combining \eqref{lem_second_moments_eqn1}-\eqref{lem_second_moments_eqn3} together.
\end{proof}

Now, we take $R=O_S$ to be the ring of $S$-integers in Lemma~\ref{lem_second_moments}. The following proposition shows a way on how to apply the upper bounds of $\pi_{K,S}(n)$ to estimate the total weighted second moments. 

\begin{proposition}\label{prop_second_moments_summation}
Let $t_1, t_2, \cdots $ be a sequence of parameters in $[0,\frac12]$. For $1\leq j\leq J$, take $\delta_j=t_n$ if $\deg P_j=n$.	Then we have
\begin{equation}\label{eqn_prop_second_moments_summation}
	\sum_{j=1}^J\frac{M_j^{(2)}}{4\delta_j(1-\delta_j)}\leq  \frac{s^2}q\sum_{n=1}^\infty \frac{1}{4t_n(1-t_n)} \frac{q\pi_{K,S}(n)}{(q^n-1)^2} \exp\set{\sum_{k=1}^n\frac{\pi_{K,S}(k)}{1-t_k} \big(\frac{3}{q^k}+ h(q^k)\big)}.
\end{equation}	
\end{proposition}

 \begin{proof} 
 First, for the summation in the exponent part of \eqref{lem_second_moments_eqn}, we have
 \begin{align}
 	\sum_{\deg P_i\leq n}\frac1{1-\delta_i}\cdot \frac1{|P_i|} &\leq \sum_{k=1}^n\frac1{1-t_k} \sum_{\deg P=k}\frac1{|P|}=\sum_{k=1}^n\frac{\pi_{K,S}(k)}{1-t_k}\cdot \frac{1}{q^k},\\
\sum_{\deg P_i\leq n}\frac{1}{1-\delta_i}\cdot h(|P_i|)&\leq \sum_{k=1}^n \frac1{1-t_k} \sum_{\deg P=k}h(|P|)= \sum_{k=1}^n \frac{\pi_{K,S}(k)}{1-t_k}\cdot h(q^k).
 \end{align}

 Then, by  \eqref{lem_second_moments_eqn} we get that
 	\begin{align}
 		&\quad\sum_{j=1}^J\frac{M_j^{(2)}}{4\delta_j(1-\delta_j)} \nonumber\\
 	&\leq  	\sum_{j=1}^J\frac{s^2}{4\delta_j(1-\delta_j)(|P_j|-1)^2} \exp\set{\sum_{i<j}\frac1{1-\delta_i}\Big(\frac{3}{|P_i|}+h(|P_i|)\Big)} \nonumber\\
 	&= \sum_{n=1}^\infty \sum_{\deg P_j=n} \frac{s^2}{4\delta_j(1-\delta_j)(|P_j|-1)^2}\exp\set{\sum_{i<j}\frac1{1-\delta_i}\Big(\frac{3}{|P_i|}+h(|P_i|)\Big)} \nonumber\\
 	&\leq \sum_{n=1}^\infty \frac{ \pi_{K,S}(n)s^2}{4t_n(1-t_n)(q^n-1)^2} \exp\set{\sum_{\deg P_i\leq n}\frac1{1-\delta_i}\Big(\frac{3}{|P_i|}+h(|P_i|)\Big)} \nonumber\\
 	&\leq  \sum_{n=1}^\infty \frac{\pi_{K,S}(n)s^2}{4t_n(1-t_n)(q^n-1)^2}\exp\set{\sum_{k=1}^n\frac{\pi_{K,S}(k)}{1-t_k} \big(\frac{3}{q^k}+ h(q^k)\big)},
 	\end{align}
which gives \eqref{eqn_prop_second_moments_summation}.
 \end{proof}

\section{Proof of Theorems~\ref{mainthm_gff} and \ref{mainthm_ff}}
 
In this section, we will apply Proposition~\ref{prop_second_moments_summation} and the upper bound \eqref{eqn_pi_upper_bound} of $\pi_{K,S}(n)$ to estimate the weighted sum of the second moments.

\begin{theorem}\label{thm_total_second_moments}
	Take $t_1=0.17, t_n=0.25$ for all $n\ge2$ in Proposition~\ref{prop_second_moments_summation}. Then for $q\ge 70$, we have
	\begin{equation}\label{thm_total_second_moments_eqn}
		\sum_{j=1}^J\frac{M_j^{(2)}}{4\delta_j(1-\delta_j)}< \frac{(82.26+18.88g)e^{0.95g}s^2}{q}.
	\end{equation}
\end{theorem}
\begin{proof}
	
 For the summation of the right hand side of \eqref{eqn_prop_second_moments_summation}, we denote by $S_1$ the term $n=1$ and $S_2$ the summation over $n\geq2$. That is,
 \begin{equation}\label{thm_total_second_moments_eqn_S1S2}
 	\sum_{n=1}^\infty \frac{1}{4t_n(1-t_n)} \frac{q\pi_{K,S}(n)}{(q^n-1)^2} \exp\set{\sum_{k=1}^n\frac{\pi_{K,S}(k)}{1-t_k} \big(\frac{3}{q^k}+ h(q^k)\big)}= \sum_{n=1}+\sum_{n\ge2}:=S_1+S_2.
 \end{equation} 
 
 Now we use the upper bounds \eqref{eqn_pi_upper_bound} of $\pi_{K,S}(n)$ to estimate $S_1$ and $S_2$ respectively. 
 
 For $S_1$,  by $\pi_{K,S}(1)\leq q+1+2g\sqrt{q}$ we have
 \begin{equation}
 	S_1\leq S_1(q):=\frac{q(q+1+2g\sqrt{q})}{4t_1(1-t_1)(q-1)^2}\cdot\exp\set{\frac{q+1+2g\sqrt{q}}{1-t_1} \big(\frac3q+h(q)\big)}. 
 \end{equation}
 Then for $q\ge70$, we have
 \begin{equation}
 	S_1\leq S_1(70)\leq \frac{0.261+0.062g}{t_1(1-t_1)}\exp\set{\frac{3.117+0.735g}{1-t_1}}.
 \end{equation}
 
 The minimum point of the function $\frac{0.261}{t_1(1-t_1)}\exp\set{\frac{3.117}{1-t_1}}$ in $[0,\frac12]$ is at $t_1=0.1732\cdots$. Taking $t_1=0.17$, we get that
 \begin{equation}\label{eqn_S1}
 	S_1\leq (79.082+18.786g)e^{0.886g}.
 \end{equation}

 For $S_2$, we take $t_n=0.25$ for all $n\ge2$. Using the fact that 
 $$\frac{1}{(q^n-1)^2}\leq \frac{1.001}{q^{2n}}, h(q^n)\leq \frac{5.002}{q^{2n}} \text{ and }  \pi_{K,S}(n)\leq \frac{q^n+1}{n}+\frac{2gq^{n/2}}{n}$$
 for all $n\geq 2$ and $q\ge70$, we get 
 \begin{multline}
 	 	S_2\leq S_2(q):=1.001\sum_{n=2}^\infty \frac{q^n+1+2gq^{n/2}}{0.75\cdot nq^{2n-1}}\cdot\exp\left\{\frac{q+1+2g\sqrt{q}}{1-t_1} \big(\frac3q+h(q)\big)\right.\\ \left. +\sum_{k=2}^n\frac{q^k+1+2gq^{k/2}}{0.75\cdot k} \big(\frac3{q^k}+\frac{5.002}{q^{2k}}\big) \right\}.
 \end{multline}
  Then for $q\ge70$, we have
  \begin{align}
  	S_2&\leq S_2(70) \leq 1.001\sum_{n=2}^\infty \frac{70^n+1+2\cdot70^{n/2}g}{0.75\cdot n\cdot70^{2n-1}} \exp\left\{\frac{3.117+0.735g}{1-t_1} +\sum_{k=2}^n\frac4k\right. \nonumber\\
  	& \qquad \left.  +\sum_{k=2}^\infty \frac1{0.75\cdot k} \big(\frac{8.002}{70^k}+\frac{5.002}{70^{2k}}\big) +2g\sum_{k=2}^\infty \frac1{0.75\cdot k}\big(\frac3{70^{k/2}}+\frac{5.002}{ 70^{3k/2}}\big)\right\} \nonumber\\
  	&\leq 1.001\sum_{n=2}^\infty \frac{70^n+1+2\cdot70^{n/2}g}{0.75\cdot n\cdot70^{2n-1}}\exp\left\{\frac{3.117+0.735g}{0.83}\right. \nonumber\\
  	& \qquad \left. +4(\log n+\frac12-\log 2) +0.002 +0.063g\right\}  \nonumber\\
  	&\leq 1.001\sum_{n=2}^\infty \frac{n^3(70^n+1+2\cdot70^{n/2}g)}{0.75\cdot 70^{2n-1}}\cdot e^{2.985+0.949g} \nonumber\\
  	&\leq (3.17+0.087g)e^{0.949g}. \label{eqn_S2}
  \end{align}
  
  Thus, adding up \eqref{eqn_S1} and \eqref{eqn_S2} gives
  \begin{equation}\label{eqn_S1S2}
  	S_1+S_2\leq (82.26+18.88g)e^{0.95g}.
  \end{equation}
 Hence \eqref{thm_total_second_moments_eqn} follows by using \eqref{eqn_S1S2} in Proposition~\ref{prop_second_moments_summation}.
\end{proof}

\begin{remark}
	The reason why we choose $t_2=0.25$ in the proof of Theorem~\ref{thm_total_second_moments} is that the term on $1-t_2$ in the exponent of \eqref{thm_total_second_moments_eqn_S1S2} is close to $\frac{3}{2(1-t_2)}$ and the minimum point of the function $\frac{1}{t_2(1-t_2)}e^{\frac{3}{2(1-t_2)}}$ in $[0,\frac12]$ is at $t_2=0.25$. As regards the terms $n\ge3$, their portion in $S_2$ is very small, hence we take $t_n=0.25, n\ge3$ as well for brevity. The numerical estimates on the infinite series, exponents and other calculations in the proofs are done by the Mathematica software.
\end{remark}

\begin{proof}[Proof of Theorem~\ref{mainthm_gff}] 
Let $\cA$ be any finite collection of congruences in $O_S$ of multiplicity $s$. If $q\geq (82.26+18.88g)e^{0.95g}s^2$, then by Theorem~\ref{thm_total_second_moments} we get that 
	\[
	\sum_{j=1}^J\frac{M_j^{(2)}}{4\delta_j(1-\delta_j)}<1.
	\]
	By Theorem~\ref{thm_distortion_method}, we conclude that $\cA$ is not a covering system of $O_S$.
\end{proof}

\begin{proof}[Proof of Theorem~\ref{mainthm_ff}]

Let $O_S=\F_q[x]$ be the polynomial ring over $\F_q$ and let $s=1$ in Proposition~\ref{prop_second_moments_summation}. Then we have the genus $g=0$ and $\pi_{K,S}(1)=q$ in this case. We follow the proof of  Theorem~\ref{mainthm_gff}.  The term $n=1$ in \eqref{thm_total_second_moments_eqn_S1S2} is equal to
\begin{equation}
 	S_1=S_1(q)=\frac{q^2}{4(q-1)^2t_1(1-t_1)}\exp\set{\frac{3+qh(q)}{1-t_1}}.
 \end{equation}
 
 For $q\ge70$, we have
 \begin{equation}
 	S_1\leq S_1(70)\leq \frac{1.03}{4t_1(1-t_1)} \exp\set{\frac{3.08}{1-t_1}}.
 \end{equation}
 Taking $t_1=0.17$ gives
 \begin{equation}
 	S_1\leq 74.62.
 \end{equation}
 
 By \eqref{eqn_S2}, we have $S_2\leq 3.17$. It follows that $S_1+S_2\leq 77.79$ and
 $$\sum_{j=1}^J\frac{M_j^{(2)}}{4\delta_j(1-\delta_j)}< \frac{77.79}{q}.$$
By Theorem~\ref{thm_distortion_method}, there is no covering system of $\F_q[x]$ with distinct moduli for $q\geq 78$. Notice that there is no prime power in $[74,78]$ but $73$ is a prime. Since $q$ is a prime power, we conclude that there is no covering system of $\F_q[x]$ with distinct moduli for $q>73$.
 \end{proof}


\end{document}